\documentclass[10pt]{amsart}
\usepackage{amsfonts,amsmath,amsthm,amssymb,color}
\usepackage[all]{xy}
\usepackage[ colorlinks=true, linkcolor=blue, citecolor=blue]{hyperref}

\DeclareMathOperator*{\colim}{colim}

\numberwithin{equation}{section}

\begin{document}

\newtheorem{theorem}{Theorem}[section]
\newtheorem{thm}[theorem]{Theorem}
\newtheorem{lemma}[theorem]{Lemma}
\newtheorem{proposition}[theorem]{Proposition}
\newtheorem{corollary}[theorem]{Corollary}

\theoremstyle{definition}
\newtheorem{definition}[theorem]{Definition}
\newtheorem{example}[theorem]{Example}

\theoremstyle{remark}
\newtheorem{remark}[theorem]{Remark}

\newenvironment{magarray}[1]
{\renewcommand\arraystretch{#1}}
{\renewcommand\arraystretch{1}}

\newcommand{\quot}[2]{
{\lower-.2ex \hbox{$#1$}}{\kern -0.2ex /}
{\kern -0.5ex \lower.6ex\hbox{$#2$}}}

\newcommand{\mapor}[1]{\smash{\mathop{\longrightarrow}\limits^{#1}}}
\newcommand{\mapin}[1]{\smash{\mathop{\hookrightarrow}\limits^{#1}}}
\newcommand{\mapver}[1]{\Big\downarrow
\rlap{$\vcenter{\hbox{$\scriptstyle#1$}}$}}
\newcommand{\liminv}{\smash{\mathop{\lim}\limits_{\leftarrow}\,}}

\newcommand{\specif}[2]{\left\{#1\,\left|\, #2\right. \,\right\}}

\renewcommand{\bar}{\overline}
\newcommand{\de}{\partial}
\newcommand{\debar}{{\overline{\partial}}}
\newcommand{\per}{\!\cdot\!}
\newcommand{\Oh}{\mathcal{O}}
\newcommand{\sA}{\mathcal{A}}
\newcommand{\sB}{\mathcal{B}}
\newcommand{\sC}{\mathcal{C}}
\newcommand{\sE}{\mathcal{E}}
\newcommand{\sF}{\mathcal{F}}
\newcommand{\sG}{\mathcal{G}}
\newcommand{\sW}{\mathcal{W}}
\newcommand{\sH}{\mathcal{H}}
\newcommand{\sI}{\mathcal{I}}
\newcommand{\sJ}{\mathcal{J}}
\newcommand{\sL}{\mathcal{L}}
\newcommand{\sM}{\mathcal{M}}
\newcommand{\sP}{\mathcal{P}}
\newcommand{\sU}{\mathcal{U}}
\newcommand{\sQ}{\mathcal{Q}}
\newcommand{\sR}{\mathcal{R}}
\newcommand{\sV}{\mathcal{V}}
\newcommand{\sX}{\mathcal{X}}
\newcommand{\sY}{\mathcal{Y}}
\newcommand{\sN}{\mathcal{N}}
\newcommand{\sS}{\mathcal{S}}

\newcommand{\Ni}{\mathcal{N}_I}
\newcommand{\Nj}{\mathcal{N}_J}
\newcommand{\bM}{\mathbf{M}}
\newcommand{\bC}{\mathbf{C}}
\newcommand{\LI}{\mathbf{L}_{\sN}}
\newcommand{\pal}{\boldsymbol{\cdot}}

\newcommand{\Aut}{\operatorname{Aut}}
\newcommand{\Mor}{\operatorname{Mor}}
\newcommand{\Def}{\operatorname{Def}}
\newcommand{\Hom}{\operatorname{Hom}}
\newcommand{\RHom}{R\mspace{-2mu}\operatorname{Hom}}
\newcommand{\Hilb}{\operatorname{Hilb}}
\newcommand{\HOM}{\operatorname{\mathcal H}\!\!om}
\newcommand{\DER}{\operatorname{\mathcal D}\!er}
\newcommand{\Spec}{\operatorname{Spec}}
\newcommand{\Sh}{\operatorname{Sh}}
\newcommand{\Der}{\operatorname{Der}}
\newcommand{\Tor}{{\operatorname{Tor}}}
\newcommand{\Ext}{{\operatorname{Ext}}}
\newcommand{\End}{{\operatorname{End}}}
\newcommand{\END}{\operatorname{\mathcal E}\!\!nd}
\newcommand{\Image}{\operatorname{Im}}
\newcommand{\Id}{\operatorname{Id}}
\newcommand{\coker}{\operatorname{coker}}
\newcommand{\res}{\operatorname{res}}
\newcommand{\rk}{\operatorname{rk}}
\newcommand{\tot}{\operatorname{tot}}
\newcommand{\cone}{\operatorname{cone}}
\newcommand{\cocone}{\operatorname{cocone}}
\newcommand{\mA}{\mathfrak{m}_{A}}
\newcommand{\g}{\mathfrak{g}}
\newcommand{\gl}{\mathfrak{gl}}
\newcommand{\GL}{\operatorname{GL}}
\newcommand{\Tr}{\operatorname{Tr}}

\newcommand{\Coder}{\operatorname{Coder}}

\renewcommand{\Hat}[1]{\widehat{#1}}
\newcommand{\dual}{^{\vee}}
\newcommand{\desude}[2]{\dfrac{\de #1}{\de #2}}

\newcommand{\A}{\mathbb{A}}
\newcommand{\N}{\mathbb{N}}
\newcommand{\R}{\mathbb{R}}
\newcommand{\Z}{\mathbb{Z}}
\renewcommand{\H}{\mathbb{H}}
\renewcommand{\L}{\mathbb{L}}
\newcommand{\proj}{\mathbb{P}}
\newcommand{\K}{\mathbb{K}\,}
\newcommand\C{\mathbb{C}}
\newcommand\Del{\operatorname{Del}}
\newcommand\D{\operatorname{D}}
\newcommand\Tot{\operatorname{Tot}}
\newcommand\Grpd{\mbox{\bf Grpd}}



\newcommand{\rh}{\rightarrow}
\newcommand{\contr}{{\mspace{1mu}\lrcorner\mspace{1.5mu}}}

\newcommand{\bi}{\boldsymbol{i}}
\newcommand{\bl}{\boldsymbol{l}}

\newcommand{\MC}{\operatorname{MC}}
\newcommand{\TW}{\operatorname{TW}}
\newcommand{\DGMod}{\operatorname{DGMod}}
\newcommand{\Ch}{\operatorname{Ch}}
\newcommand{\Ho}{\operatorname{Ho}}
\newcommand{\PM}{\mathbf{Mod}}
\newcommand{\Mod}{\mathbf{Mod}}
\newcommand{\DGPM}{\mathbf{Mod}^{\ast}}
\newcommand{\QCoh}{\mathbf{QCoh}}
\newcommand{\DGSch}{\mathbf{DGSch}}
\newcommand{\DGAff}{\mathbf{DGAff}}

\newcommand{\Set}{\mathbf{Set}}
\newcommand{\Art}{\mathbf{Art}}
\newcommand{\Alg}{\mathbf{Alg}}
\newcommand{\DGArt}{\mathbf{DGArt}}
\newcommand{\CDGA}{\mathbf{CDGA}}
\newcommand{\CGA}{\mathbf{CGA}}
\newcommand{\DGLA}{\mathbf{DGLA}}
\newcommand{\solose}{\Rightarrow}
\newcommand{\PSI}{\Psi\mathbf{Sch}_I(\mathbf{M})}
\newcommand{\PSJ}{\Psi\mathbf{Sch}_J(\mathbf{M})}
\newcommand{\Sym}{\mbox{Sym}}

\title[Deformations of polystable sheaves on surfaces]{Deformations of polystable sheaves on surfaces: quadraticity implies formality}

\date{February 17, 2019}

\author[R. Bandiera]{Ruggero Bandiera}
\address{\newline
Universit\`a degli studi di Roma La Sapienza,\hfill\newline
Dipartimento di Matematica \lq\lq Guido
Castelnuovo\rq\rq,\hfill\newline
P.le Aldo Moro 5,
I-00185 Roma, Italy.}
\email{bandiera@mat.uniroma1.it}

\author[M. Manetti]{Marco Manetti}
\email{manetti@mat.uniroma1.it}
\urladdr{www.mat.uniroma1.it/people/manetti/}

\author[F. Meazzini]{Francesco Meazzini}
\email{meazzini@mat.uniroma1.it}


\maketitle

\begin{abstract}
We study relations between the quadraticity of the Kuranishi family of a coherent sheaf on a complex projective scheme
and the formality of the DG-Lie algebra of its derived endomorphisms. In particular,  we prove that   
for a polystable coherent sheaf of a smooth complex projective surface
the DG-Lie algebra of derived endomorphisms is formal if and only if the Kuranishi family  is quadratic.

\end{abstract}

\tableofcontents

\section{Introduction}

Recall that for a coherent sheaf $\sF$ over a projective scheme, we say that the Kuranishi family  is quadratic if the base space of the Kuranishi family (i.e., the semiuniversal deformation) of $\sF$ is a quadratic singularity. One of the main goals of this paper is to prove the following result.

\begin{theorem}\label{thm.maintheorem1} 
Let $\sF$ be a polystable coherent sheaf of a smooth complex projective surface.
Then the differential graded (DG) Lie algebra $\RHom(\sF,\sF)$ of derived endomorphisms is formal if and only if the  Kuranishi family of $\sF$ is  quadratic.
\end{theorem}

The initial motivation was to provide a proof of the so called Kaledin-Lehn formality conjecture for polystable sheaves on 
projective K3 surfaces as a consequence of the quadraticity theorem of Kuranishi space proved by Arbarello and Sacc\`a \cite{AS} and by Yoshioka~\cite{yoshioka}. 

The above mentioned formality conjecture has been explicitly stated by Zhang in~\cite{Zh12}, where it is attributed to Kaledin and Lehn, and it was recently proved by Budur and Zhang \cite{budur}; we also refer to 
\cite{budur} for a detailed description of the formality conjecture and its implication in the theory of moduli spaces 
of sheaves on K3 surfaces. However, the proof of Budur and Zhang is done along a completely different approach, and  
our result  applies to every smooth projective surface, not necessarily K3.
It is worth mentioning that in a very recent paper~\cite{BaMaMe} we approached the problem from an alternative point of view finally enlarging the class of surfaces on which the formality conjecture holds: K3 surfaces, Abelian surfaces, Enriques surfaces and bielliptic surfaces.
In both our papers, the approaches involve techniques of $L_{\infty}$-algebras, in particular we investigate derived endomorphisms as a DG-Lie algebra instead of as an associative DG algebra; it is important to point out that the formality is a stronger result in the associative case.

Since the automorphisms group of a polystable sheaf is a product of groups $\GL_n(\C)$ (see e.g. \cite{HL}), 
Theorem~\ref{thm.maintheorem1} is an immediate consequence of the following more general result. 

\begin{theorem}\label{thm.maintheorem2} 
Let $\sF$ be a coherent sheaf of finite projective dimension on a complex projective scheme $X$. Assume that:
\begin{enumerate}

\item the algebraic group $\Aut_X(\sF)$ of automorphisms of $\sF$ is linearly reductive;

\item the trace map $\Tr\colon \Ext^i_X(\sF,\sF)\to H^i(\Oh_X)$ is injective for every $i\ge 3$.
\end{enumerate}
Then the DG-Lie algebra $\RHom(\sF,\sF)$ of derived endomorphisms is formal if and only if the Kuranishi family of $\sF$ is quadratic.
\end{theorem}

Here by saying that $\sF$ is of finite projective dimension we mean that the projective dimensions of its stalks are (uniformly) bounded.

Recall that the DG-Lie algebra of derived endomorphisms of $\sF$ controls the deformation theory of $\sF$ (see e.g. \cite{AS,budur,FIM,FM}) in the usual way, via Maurer-Cartan 
modulus gauge action, see e.g. \cite{GoMil1,Man}, and its cohomology compute the derived functors of $\Hom_X$, i.e., 
\[ H^i(\RHom(\sF,\sF))=\Ext^i_X(\sF,\sF)\,.\]

The ``only if'' implication in Theorem~\ref{thm.maintheorem2} 
is a completely standard consequence of the homotopy invariance of the Kuranishi family: a detailed proof of this fact is given  in \cite{GoMil1,GoMil2},  where it is used in order to prove that 
certain  moduli spaces of representations of the fundamental group of a compact K\"{a}hler manifold have quadratic singularities, cf. also \cite{MartPadova} and references therein.

The idea of proof of Theorem~\ref{thm.maintheorem2} is the following. 
Assume for simplicity that $\Ext^i_X(\sF,\sF)=0$ for every $i\ge 3$ and 
suppose first heuristically that $\Aut_X(\sF)$ is a discrete group (clearly this is true only for $\sF=0$), 
then $H^i(\RHom(\sF,\sF))=0$ for every $i\not=1,2$; under these assumptions 
quadraticity implies formality by the well known fact that 
the Kuranishi moduli space uniquely determines the minimal $L_{\infty}[1]$-model of 
$\RHom(\sF,\sF)$ up to (non-canonical) isomorphism; for reference purposes we give a proof of  this fact in Subsection \ref{subsec.formalityquadra}. 

In the general case, the first step is to show that there exists  a DG-Lie representative $L$ 
of $\RHom(\sF,\sF)$ equipped with a rational algebraic action of the group $\Aut_X(\sF)$, compatible with the natural induced action on 
deformations of $\sF$ and with the adjoint action on the graded vector space 
$\Ext^*_X(\sF,\sF)$. The second step is to show that the above DG-Lie algebra $L$ can be chosen as the semidirect product of 
the Lie algebra $\Ext^0_X(\sF,\sF)$ and a DG-Lie algebra $K$ concentrated in strictly positive degrees. By general facts $L$ and $K$ have the same Kuranishi family, hence 
$L$ is quadratic iff $K$ is quadratic iff $K$ is formal.
The last step is to use the reductivity of $\Aut_X(\sF)$ and some argument of \cite{Man4} for proving that $K$ is formal if and only if $L$ is formal.

The construction of an equivariant representative of derived endomorphisms requires some care, since in general 
the rationality of  (infinite dimensional) representations is not preserved by $\Hom$ functors:
for instance, given a rational representation $V$ of the algebraic group $\GL_1(\C)=\C^*$, 
the dual representation $\Hom_{\C}(V,\C)$  
is rational if and only if $V$ has a finite number of isotypic components.
This suggests to restrict our attention to rational finitely supported representations: 
a detailed introduction  to  finitely supported representations is given in Section~\ref{sec.finitelysupp}; here we only anticipate that 
a semisimple representation is finitely supported if and only if  it has a finite number of isotypic components.

This restriction forces to avoid injective resolutions of $\sF$ and to work, following \cite{FM}, with local projective resolutions 
with respect to a finite affine cover of $X$. Then at the end of Section~\ref{sec.three} we shall prove the 
following result.

\begin{theorem}\label{thm.maintheorem3} 
Let $\sF$ be a coherent  sheaf of finite projective dimension on a complex projective scheme $X$. 
Then there exists a  DG-Lie algebra $L$, representing $\RHom(\sF,\sF)$, together with a 
DG-Lie action of $\Aut_X(\sF)$ on $L$ (i.e., a morphism of groups
 $\Aut_X(\sF)\to \Aut_{DG-Lie}(L)$) such that:
\begin{enumerate}

\item the action of $\Aut_X(\sF)$ on every homogeneous component $L^i$, $i\in \Z$, is rational algebraic and finitely supported;

\item  there exists an invariant  Lie subalgebra
$\mathfrak{g}\subset Z^0(L)$, canonically isomorphic to  $\Ext^0_X(\sF,\sF)$ via 
the projection $Z^0(L)\to H^0(L)$; 

\item via the above canonical isomorphism $\mathfrak{g}\cong\Ext^0_X(\sF,\sF)$ the adjoint action of $\mathfrak{g}$ on $L$ is
the same as the infinitesimal action of $\Aut_X(\sF)$.

\end{enumerate}
\end{theorem}

It is interesting to point out that, as a consequence  of Theorem~\ref{thm.maintheorem3}, 
we have that if  
$\Aut_X(\sF)$ is linearly reductive, then $L$ admits an equivariant Hodge decomposition and the usual Kuranishi argument, see e.g. \cite[Sec. 4]{NijRich64} and \cite[Sec. 2]{GoMil2}, gives an equivariant Kuranishi family.

\section{Finitely supported  and rational representations}
\label{sec.finitelysupp}

Let $G$ be a group and $\K$ a fixed field. For the purposes of this section we will only need $\K$ to be infinite, while from Section~\ref{sec.three} on we shall assume $char(\K)=0$. Recall (see e.g., \cite{procesi}) that a $\K$-linear representation $G\to GL(V,\K)$, equivalently a left $\K[G]$-module, where $V$ is a (possibly infinite dimensional) vector space, is called:

\begin{itemize}
\item \emph{trivial} if  $gv=v$ for every $g\in G$, $v\in V$;

\item \emph{irreducible} if the only $G$-invariant subspaces of $V$ are $0,V$;

\item \emph{semisimple} if $V$ is a direct sum of irreducible representations.

\end{itemize}

Given two representations $V,W$ we shall denote by $\Hom_{\K}(V,W)^G\subseteq \Hom_{\K}(V,W)$ the  
vector subspace of $G$-equivariant linear maps. By definition, $\Hom_{\K}(V,W)^G$ is the set of morphisms
from $V$ to $W$ in the category of $\K$-linear representations of $G$.

\begin{definition}\label{def.isotypic}
Denote by $\K[G]^\vee$ the set of isomorphism classes of irreducible representations of $G$. Given 
$\alpha\in \K[G]^\vee$,  
the \emph{isotypic component} $V_\alpha$ of a representation $V$ is the sum of all
irreducible $G$-invariant subspaces $H\subset V$ in the class $\alpha$. 
\end{definition}

Notice that if $H$ is an irreducible representation in the class $\alpha$, then $V_{\alpha}\not=0$ if and only if 
$\Hom_{\K}(H,V)^G\not=0$.

\begin{lemma}\label{lem.sottomoduloisotipica} 
In the above situation, if $\alpha\in \K[G]^\vee$,
$K\subset V$  is an irreducible submodule and $K\cap V_\alpha\not=0$, then $K$ is in the class 
$\alpha$.
\end{lemma}

\begin{proof} 
If $K\cap V_\alpha\not=0$, then  
there exists a finite number 
of irreducible submodules $H_1,\ldots,H_n$ in the class $\alpha$ such that 
$K\cap \sum H_i\not=0$. Thus there exists an index $1\le r\le n$ such that 
\[ K\cap \sum_{i<r}H_i=0,\qquad K\cap \sum_{i\le r}H_i\not=0,\]
and then $K\subset \sum_{i\le r}H_i$. A fortiori $\sum_{i<r}H_i\not=\sum_{i\le r}H_i$ and
the restrictions to $K$ and $H_r$ of the projection to the quotient give two non-trivial maps
\[ \xymatrix{&H_r\ar[d]^-{p}\\
K\ar[r]^-{i}&\dfrac{\;\sum_{i\le r}H_i\;}{\;\;\sum_{i<r}H_i\;\;}}\]
with $p$ surjective. Since both $K$ and $H_r$ are irreducible it follows that both $i$ and $p$ are isomorphisms, and hence $K$ belongs to the class $\alpha$.
\end{proof}

\begin{lemma}\label{lem.isotipicapertriviale} 
In the above situation, if $W$ is a trivial representation and 
$M= V\otimes W$, then $M_{\alpha}= V_{\alpha}\otimes W$ for every class $\alpha\in \K[G]^\vee$.
\end{lemma}

\begin{proof} If 
$H\subset V$ is an  irreducible submodule in the class $\alpha$, then $H\otimes W$ is a direct sum of copies of $H$, 
therefore $H\otimes W\subset M_{\alpha}$ and this implies $V_{\alpha}\otimes W\subset M_{\alpha}$.
Conversely, if 
$H\subset V\otimes W$ is  a  irreducible submodule  in the class
$\alpha$, we need to prove that $h\in   V_{\alpha}\otimes W$ for every $h\in H$.
Writing $h=\sum_{i=1}^n v_i\otimes w_i$ with $v_i\in V\setminus\{0\}$, $w_i\in W$ and $w_1,\ldots,w_n$ linearly independent, for every $i$ we can find 
a linear map $f_i\colon W\to \K$ such that $f(w_i)=1$ and $f(w_j)=0$ for $j\not=i$; 
the images of 
$H$ under the $G$-equivariant maps 
\[V\otimes W\to V,\qquad v\otimes w\mapsto f_i(w)v,\]
are non trivial and hence isomorphic to $H$. Therefore $v_i\in V_{\alpha}$ for every $i$ and 
$h\in V_{\alpha}\otimes W$.
\end{proof}

\begin{definition}
The \emph{support} $S_G(V)\subset \K[G]^\vee$ of a representation $V$ of $G$ is the set of isomorphism classes of irreducible representations
$H$ such that  
$\Hom_{\K}(H,V/W)^G\not=0$ for some $G$-invariant subspace $W\subset V$.
A representation $V$ is called \emph{finitely supported} if its support is finite.  
\end{definition}

If $U\subset V$ is a $G$-invariant subspace then   
it is obvious that $S_G(V/U)\subseteq S_G(V)$ and it is easy to see that $S_G(U)\subseteq S_G(V)$: 
in fact if $W\subset U$ is an invariant subspace and $H$ an irreducible representation such that 
$\Hom_{\K}(H,U/W)^G\not=0$ 
then $U/W\subset V/W$ and therefore $\Hom_{\K}(H,V/W)^G\not=0$. Moreover
$S_G(V\oplus W)=S_G(V)\cup S_G(W)$: it is clear that  $S_G(V)\cup S_G(W)\subset S_G(V\oplus W)$; conversely if 
$V\oplus W\xrightarrow{p}U$ is a surjective morphism of representations and $K\subset U$ is irreducible, then 
either $K\subset p(V)$, and then $K\in S_G(V)$, or 
$K\subset U/p(V)$, and then $K\in S_G(W)$.

\begin{lemma}\label{lem.calcolosupporto} 
Let $V$ be a representation of $G$ and $\alpha\in \K[G]^\vee$ such that 
$V_{\alpha}\not=0$. Then $S_G(V)=\{\alpha\}\cup S_G(V/V_\alpha)$. 
\end{lemma}

\begin{proof} The only nontrivial inclusion to prove is  
$S_G(V)\subset \{\alpha\}\cup S_G(V/V_\alpha)$. 
Let $p\colon V\to W$ be a surjective morphism of representations and let 
$K\subset W$ be an irreducible submodule; we need to prove that if 
$K$ is not in the class $\alpha$ then its isomorphism class belongs to $S_G(V/V_\alpha)$.
Since $W/p(V_{\alpha})$ is a quotient of $V/V_\alpha$, it is sufficient to prove that 
$K\not\subset p(V_{\alpha})$ and therefore that 
$\Hom_{\K}(K,W/p(V_{\alpha}))^G\not=0$.

By definition $V_{\alpha}=\sum_{i}H_i$, where every $H_i\subset V$ is irreducible in the class $\alpha$, 
thus $p(V_\alpha)=\sum_{i}p(H_i)$ is contained in the isotypic component $W_{\alpha}$. 
If $K\subset p(V_{\alpha})\subset W_\alpha$ then by Lemma~\ref{lem.sottomoduloisotipica} the isomorphism class of $K$ is $\alpha$. 

\end{proof}

\begin{corollary} Let $V$ be a finite dimensional representation and $W$ a trivial representation.
Then $S_G(V\otimes W)=S_G(V)$ is finite.
\end{corollary}

\begin{proof} We prove the result by induction on the dimension of $V$.
If $V=0$ there is nothing to prove; 
if $V\not=0$, since $\dim V<\infty$ 
there exists an irreducible submodule $H\subset V$,
and then also a nontrivial isotypic component $V_{\alpha}\subset V$.
Denoting  $M=V\otimes W$, by Lemma~\ref{lem.isotipicapertriviale} we have $M_{\alpha}=V_\alpha\otimes W$. 
By Lemma~\ref{lem.calcolosupporto} and the inductive assumption we have that 
$S_G(V/V_{\alpha}\otimes W)=S_G(V/V_{\alpha})$ is finite and then
\[ S_G(M)=\{\alpha\}\cup S_G(M/M_{\alpha})=\{\alpha\}\cup S_G(V/V_{\alpha}\otimes W)=
\{\alpha\}\cup S_G(V/V_{\alpha})=S_G(V)\,\]
is also finite.
\end{proof}

The following result is clear.

\begin{lemma}\label{lemma.finsup}
A representation $V$ is semisimple and  finitely supported if and only if there exists a finite number of irreducible representations $H_1,\ldots,H_n$ and trivial representations 
$W_1,\ldots,W_n$ such that 
\[ V=(H_1\otimes W_1)\oplus\cdots\oplus (H_n\otimes W_n)\,.\]
\end{lemma}

\begin{remark}\label{rmk.subobjectquotient}
Assume now that $G$ is an algebraic group over an infinite field $\K$. 
Recall that a representation $\phi\colon G\to GL(V,\K)$ is called \emph{rational} if every vector $v\in V$ belongs to a finite dimensional $G$-invariant subspace 
$U\subset V$ and the group homomorphism $G\to GL(U,\K)$ is algebraic, cf. \cite{procesi}.
Notice that:
\begin{enumerate}

\item  every irreducible rational representation is finite dimensional;

\item let $U\subset V$ be an invariant subspace of a representation, 
by general representation theory  if $V$ is semisimple (resp.: rational), then also 
$U$ and $V/U$ are semisimple (resp.: rational); 

\item if $V$ is a finite dimensional rational representation and $W$ is a trivial representation, then $V\otimes W$ is a rational representation that is finitely supported for every subgroup of $G$.
\end{enumerate}
\end{remark}


\bigskip

\section{A rational and finitely supported model for $\RHom(\sF,\sF)$}
\label{sec.three}

\subsection{Preliminaries and notation}

Our aim is to study derived endomorphisms of a quasi-coherent sheaf $\sF$ on a projective scheme $X$ over a field $\K$ of characteristic $0$.
Even if most of the following results hold under mild assumptions, for the sake of simplifying the exposition we decided to fix the above hypothesis on $X$ throughout all this section.
We begin by introducing the geometric framework we shall work with, and by briefly recalling the results needed.

The main tool we will make use of is the category $\Mod(A_{\pal})$ of modules over (the diagram $A_{\pal}$ representing) the scheme $X$. Fix an open affine covering $\sU=\{U_h\}$ for $X$, and consider its nerve
\[ \sN=\{ \alpha=\{h_0,\dots,h_k\}\,\vert\,U_{\alpha}=U_{h_0}\cap\dots \cap U_{h_k}\neq\emptyset \} \]
which carries a \emph{degree function} $\deg\colon\sN\to\N$ defined by $\deg(\{h_0,\dots,h_k\}) = k$; moreover $\sN$ can be seen as a poset by setting $\alpha\leq\beta$ if and only if $\alpha\subseteq\beta$.
Then to $X$ it is associated the diagram $A_{\pal}$ of commutative unitary $\K$-algebras defined as
\[ A_{\pal}\colon \sN\to \Alg_{\K} \; ,\qquad \qquad \alpha\mapsto A_{\alpha}=\Gamma(U_{\alpha},\Oh_X) \]
where the map $A_{\alpha}\to A_{\beta}$ is induced by the inclusion $U_{\beta}=\Spec(A_{\beta})\hookrightarrow \Spec(A_{\beta})=U_{\alpha}$ for any $\alpha\leq\beta$ in $\sN$.

For every unitary commutative ring $A$ we shall denote by $\DGMod(A)$ the category of cochain complexes of $A$-modules, equipped with the projective model structure \cite[Sec. 2.3]{Hov99}.

\begin{remark}
For the moment we need neither the covering nor the nerve to be finite, so that we could keep working on an arbitrary $\K$-scheme $X$ which is only assumed to be separated, which is sufficient for $A_{\pal}$ to be well-defined (i.e. intersections of affines are affines).
\end{remark}

\begin{definition}\label{def.pseudo-module}
An $A_{\pal}$\textbf{-module} $\mathcal{F}$ over $X$, with respect to the fixed covering $\sU$, consists of the following data:
\begin{enumerate}
\item an object $\sF_{\alpha}\in\DGMod(A_{\alpha})$, for every $\alpha\in\sN$,
\item a morphism $f_{\alpha\beta}\colon \sF_{\alpha}\otimes_{A_{\alpha}}A_{\beta}\to \sF_{\beta}$ in $\DGMod(A_{\beta})$, for every $\alpha\leq\beta$ in $\sN$,
\end{enumerate}
satisfying the \emph{cocycle condition} $f_{\beta\gamma}\circ\left(f_{\alpha\beta}\otimes\Id_{A_{\gamma}}\right) = f_{\alpha\gamma}$, for every $\alpha\leq\beta\leq\gamma$ in $\sN$.
\end{definition}

Similar notions were considered in~\cite{EE,FK,GS,Sto}. Taking advantage of the standard projective model structure on DG-modules, the category $\Mod(A_{\pal})$ has been endowed with a cofibrantly generated model structure, see~\cite[Theorem 3.9]{FM}, where weak equivalences and fibrations are detected pointwise. In order to work with quasi-coherent sheaves, we need a (homotopical) version of quasi-coherence for $A_{\pal}$-modules: $\sF\in\Mod(A_{\pal})$ is called \emph{quasi-coherent} if all the maps $f_{\alpha\beta}$ introduced above are quasi-isomorphisms, see~\cite[Definition 3.12]{FM}.

Now, denote by $\Ho(\QCoh(A_{\pal}))$ the category of quasi-coherent $A_{\pal}$-modules localised with respect to the weak equivalences. Then there exists an equivalence of triangulated categories
\[ \Psi\colon \D(\QCoh(X)) \to \Ho(\QCoh(A_{\pal})) \]
where $\D(\QCoh(X))$ denotes the unbounded derived category of quasi-coherent sheaves on $X$, see~\cite[Theorem 5.7]{FM}. A partial result in this direction was previously proven in~\cite[Proposition 2.28]{BF}.

\begin{remark}
The functor $\Psi$ introduced above may be easily defined as follows. Given a complex $\sG^{\ast}$ of quasi-coherent sheaves on $X$, with a little ambiguity of notation we shall denote by $\Psi\sG^{\ast}$ the $A_{\pal}$-module defined by $(\Psi\sG^{\ast})_{\alpha}=\sG^{\ast}(U_{\alpha})$, where for any $\alpha\leq\beta$ in $\sN$ the map
\[ (\Psi\sG^{\ast})_{\alpha}\otimes_{A_{\alpha}}A_{\beta} \to (\Psi\sG^{\ast})_{\beta} \]
is the natural isomorphism given degreewise by the restriction maps of the sheaves $\sG^k$, $k\in\mathbb{Z}$.
Then define 
\[ \Psi[\sG^{\ast}] = [\Psi\sG^{\ast}] \in \Ho(\QCoh(A_{\pal})) \]
for any $[\sG^{\ast}] \in \D(\QCoh(X))$.
\end{remark}

The aim of the next section is to describe $\RHom(\sF,\sF)$ for a given quasi-coherent sheaf, in terms of a cofibrant replacement of $\Psi\sF$, where $\sF$ has to be thought of as a complex concentrated in degree $0$.

\bigskip

\subsection{Derived endomorphisms of complexes of locally free sheaves}
\label{section.derivedEnd}

Throughout this subsection we shall consider a fixed \emph{bounded above} complex of locally free sheaves $\sE^{\ast}$. The aim is to give an explicit description of derived endomorphisms $\RHom(\sE^{\ast},\sE^{\ast})$. Following~\cite{FM}, the idea is to replace the $A_{\pal}$-module $\Psi\sE^{\ast}$ by a cofibrant replacement, whose endomorphisms in the homotopy category $\Ho(\QCoh(A_{\pal}))\simeq \D(\QCoh(X))$ will represent the desired derived endomorphisms of the given complex $\sE^{\ast}$.

Take an open affine cover $\{U_h\}_{h\in H}$ and let $\sN$ be its nerve. Fix a total order on $H$, and for each $\alpha\in\sN$ denote by $\Delta_{\alpha}$ the oriented abstract simplicial complex whose faces are defined as the non-empty subsets of $\alpha$. The associated chain complex will be denoted by $C_{\ast}(\Delta_{\alpha})$; recall that for any $r\in\Z$ the rank of the free $\Z$-module $C_{r}(\Delta_{\alpha})$ counts the number of $r$-faces of $\Delta_{\alpha}$, i.e. $\rk \left(C_r(\Delta_{\alpha})\right)={\binom{\deg(\alpha)+1}{r+1}}$. More precisely:
\[ \begin{aligned}
C_r(\Delta_{\alpha})=\bigoplus_{\substack{\beta\leq\alpha \\ \vert\beta\vert =r+1}}\Z\cdot\beta &\longrightarrow \bigoplus_{\substack{\gamma\leq\alpha \\ \vert\gamma\vert =r}}\Z\cdot\gamma= C_{r-1}(\Delta_{\alpha}) \\
\beta=(i_0,\dots,i_r) &\mapsto \sum_{j=0}^r(-1)^j(i_0,\dots,\widehat{i_j},\dots,i_r)
\end{aligned} \]
Concerning its homology we have $H_0(C_{\ast}(\Delta_{\alpha};\Z))=\Z$ and $H_r(C_{\ast}(\Delta_{\alpha};\Z))=0$ for every $r\neq 0$.

We now introduce the $A_{\pal}$-module $C^{\ast}_{\pal}(\sE^{\ast})$, which we will prove to be the above mentioned cofibrant replacement of $\Psi\sE^{\ast}$. In order to avoid possible confusion we shall always denote by $\pal$ the dependence on $\sN$ and respectively by $\ast$ the degrees of the complexes; moreover, following the standard notation we shall use labels on the top to denote degrees of \emph{cochain} complexes and on the bottom for \emph{chain} complexes. Therefore, for any $\alpha\in\sN$ we can define the cochain complex $\left(\check{C}^{\ast}(\Delta_{\alpha}), \check{\partial}\right)$ as
\[ \check{C}^{r}(\Delta_{\alpha}) = C_{-r}(\Delta_{\alpha}) \; ; \qquad \qquad \check{\partial}^r = \partial_{-r}\colon \check{C}^{r}(\Delta_{\alpha}) \to \check{C}^{r+1}(\Delta_{\alpha})\]
and eventually $C^{\ast}_{\alpha}(\sE^{\ast})= \check{C}^{\ast}(\Delta_{\alpha})\otimes_{\Z}\sE^{\ast}(U_{\alpha})$. By definition, the cohomology of $C^{\ast}_{\alpha}(\sE^{\ast})$ is non trivial only in degree zero: $H^0\left( C^{\ast}_{\alpha}(\sE^{\ast}) \right) \cong \sE(U_{\alpha})$.

Our next goal is to prove that the $A_{\pal}$-module $C^{\ast}_{\pal}(\sE^{\ast})$ is cofibrant in $\Mod(A_{\pal})$. By \cite[Theorem 3.9]{FM}, this is equivalent to prove that for every choice of $\alpha\in\sN$ the natural \emph{latching} map
\[ l_{\alpha}\colon \colim_{\gamma<\alpha}\left(C^{\ast}_{\gamma}(\sE^{\ast})\otimes_{A_{\gamma}}A_{\alpha}\right)\to C^{\ast}_{\alpha}(\sE^{\ast}) \]
is a cofibration of DG-modules over $A_{\alpha}$. To this aim, consider the short exact sequence
\[ 0\to \colim_{\gamma<\alpha} \check{C}^{\ast}(\Delta_{\gamma}) \xrightarrow{\iota_{\alpha}} \check{C}^{\ast}(\Delta_{\alpha}) \to K_{\alpha}\to 0 \]
and notice that $K_{\alpha}=\coker(\iota_{\alpha})$ is a cochain complex concentrated in degree $-\deg(\alpha)$: $K_{\alpha}~=~\Z[\deg(\alpha)]\cdot\alpha$. Now, the latching map
\[ \l_{\alpha} \colon \colim_{\gamma<\alpha}\left(C^{\ast}_{\gamma}(\sE^{\ast})\otimes_{A_{\gamma}}A_{\alpha}\right) \cong \colim_{\gamma<\alpha}\left( \check{C}^{\ast}(\Delta_{\gamma}) \right)\otimes_{\Z}\sE^{\ast}(U_{\alpha}) \longrightarrow \check{C}^{\ast}(\Delta_{\alpha})\otimes_{\Z}\sE^{\ast}(U_{\alpha}) = C^{\ast}_{\alpha}(\sE^{\ast}) \]
equals $\iota_{\alpha}\otimes \Id_{\sE^{\ast}(U_{\alpha})}$ and hence it has cofibrant cokernel (because we assumed $\sE^{\ast}$ to be bounded above) and it is degreewise split injective, hence by~\cite[Lemma 2.3.6]{Hov99} it is a cofibration of differential graded $A_{\alpha}$-modules as required.

In order to show that $C^{\ast}_{\pal}(\sE^{\ast})$ is a cofibrant replacement of $\Psi\sE^{\ast}$ in $\Mod(A_{\pal})$, we are only left with the proof of the existence of a trivial fibration $C^{\ast}_{\pal}(\sE^{\ast})\to \Psi\sE^{\ast}$ of $A_{\pal}$-modules. To this aim, it is sufficient to consider the natural projections $C^{\ast}_{\alpha}(\sE^{\ast})\to H^0\left( C^{\ast}_{\alpha}(\sE^{\ast}) \right) = \sE(U_{\alpha})$ for every $\alpha~\in~\sN$, which by functoriality give the desired map of $A_{\pal}$-modules. Moreover, again by~\cite[Theorem 3.9]{FM} the morphism $C^{\ast}_{\pal}(\sE^{\ast})\to \Psi\sE^{\ast}$ is a trivial fibration of $A_{\pal}$-modules being pointwise a surjective quasi-isomorphism.
We conclude that the derived endomorphisms $\RHom_X(\sE^{\ast},\sE^{\ast})$ are represented by the DG-Lie algebra $\Hom^{\ast}_{A_{\pal}}\left(C^{\ast}_{\pal}(\sE^{\ast}),C^{\ast}_{\pal}(\sE^{\ast})\right)$, cf.~\cite[Theorem 6.4]{FM}.

\bigskip

\subsection{A locally free $\Aut_X(\sF)$-equivariant resolution for $\sF$}
\label{subsec.locfreeres}

Throughout this subsection $\sF$ will be a fixed coherent sheaf of finite projective dimension on $X$. Moreover, we shall denote by $\Aut_{X}(\sF)$ the group of automorphisms of $\sF$.

Recall that for a given  global section $s\in H^0(X,\sF)$ there exists a unique map of $\Oh_X$-modules $\varphi\colon\Oh_X\to\sF$ defined by $f\mapsto fs$. Notice that $s=\varphi(1)$. Under our hypothesis on $X$ and $\sF$ Serre's theorem applies, see e.g.~\cite[Theorem~5.17]{Har}, so that for $n\in\N$ sufficiently large the sheaf $\sF(n)$ is generated by global sections, i.e. the map $H^0(X,\sF(n)) \otimes \Oh_X  \to \sF(n)$ is surjective, where $H^0(X,\sF(n))$ has finite dimension. Hence, since tensoring by $\Oh_X(-n)$ is an exact functor, we obtain a surjective morphism
\[ \varphi\colon H^0(X,\sF(n)) \otimes \Oh_X(-n) \to \sF \]
of coherent $\Oh_X$-modules, which stays surjective when computed on any affine $U\subseteq X$:
\[ \varphi_U\colon H^0(X,\sF(n)) \otimes \Oh_X(-n)(U) \to \sF(U) \to 0 \; ; 	\qquad \qquad 	 \sum_i s_i\otimes f_i \mapsto \sum_i s_i\vert_U \cdot f_i \; . \]

In particular, we have an injective homomorphism of groups $\Aut_{X}(\sF)\to \GL(H^0(X,\sF(n)))$ together with a short exact sequence of coherent sheaves
\[ 0\to \sG\to H^0(X,\sF(n))\otimes \Oh_X(-n)\xrightarrow{\varphi} \sF\to 0\]
and then $\Aut_{X}(\sF)$ is the stabilizer of $\sG$ under the action of $\GL(H^0(X,\sF(n)))$ on the sheaf 
$H^0(X,\sF(n))\otimes \Oh(-n)$.
Notice that the action is trivial on $\Oh_X(-n)(U)$ since each element of $\Aut_{X}(\sF)$ is $\Oh_X$-linear and  
$\varphi_U$ is $\Aut_{X}(\sF)$-equivariant.

\begin{remark}
Notice that the above construction is clearly functorial in the sense that if $\alpha\leq\beta$ in the nerve $\sN$, then we have a commutative square:
\[ \xymatrix{ 	H^0(X,\sF(n))\otimes \Oh(-n)(U_{\alpha})\ar@{->}[d] \ar@{->}[r] & \sF(U_{\alpha})\ar@{->}[d] \\
H^0(X,\sF(n))\otimes \Oh(-n)(U_{\beta}) \ar@{->}[r] & \sF(U_{\beta}) 	} \]
where the horizontal arrows are surjective and the vertical arrows are the restriction maps.
\end{remark}

%

\begin{lemma}\label{lemma.coherentRFS}
Let $\sF$ be a coherent sheaf on a projective $\K$-scheme $X$. Then for every open subset $U\subset X$, the space of sections $\Gamma(U,\sF)$ is a rational representation of $\Aut_X(\sF)$ that is finitely supported for every subgroup $G\subseteq\Aut_X(\sF)$.
\begin{proof}
If $U=\bigcup U_i$ is a finite open affine cover we have 
$\Gamma(U,\sF)\subset \oplus_i \Gamma(U_i,\sF)$ and then by Remark~\ref{rmk.subobjectquotient} it is not restrictive to prove the statement assuming $U$ to be an open affine subset. 
We have already proved that  there exists $n\in\N$ sufficiently large such that 
for every affine open subset $U\subset X$ there exists a surjective equivariant map 
\[ H^0(X,\sF(n))\otimes \Gamma(U,\Oh(-n))\to \Gamma(U,\sF)\,.\]
The conclusion follows by Remark~\ref{rmk.subobjectquotient}.
\end{proof}
\end{lemma}

\begin{proposition}\label{prop.resolution}
Let $X$ be a projective $\K$-scheme, and let $\sF$ be a coherent sheaf of finite projective dimension on $X$. Then there exists a finite locally free $\Aut_{X}(\sF)$-equivariant resolution $\sE^{\ast}\to \sF$ such that for any open subset $U\subseteq X$ the complex $\Gamma(U,\sE^{\ast})$ is a degreewise rational representation that is finitely supported for every subgroup $G\subseteq \Aut_{X}(\sF)$.
\begin{proof}
Choose $n$ sufficiently large giving a short exact sequence
\[ 0\to \sF_1 \to H^0(X,\sF(n)) \otimes \Oh_X(-n) \xrightarrow{\varphi} \sF \to 0 \]
where each map is $\Aut_{X}(\sF)$-equivariant.
Since $\sF_1$ is a coherent sheaf, we can reproduce the same argument and by the hypothesis on the projective dimension of $\sF$ we obtain a resolution of $\sF$ of the form
\begin{multline*}
\quad 0\to \sF_{k} \to H^0(X,\sF_{k-1}(n_{k-1})) \otimes \Oh_X(-n_{k-1}) \to \cdots\\
\cdots \to H^0(X,\sF(n_0)) \otimes \Oh_X(-n_0) \to \sF \to 0\quad\end{multline*}
where $\sF_k$ is locally free and each map is $\Aut_{X}(\sF)$-equivariant.  To conclude, applying the functor $\Gamma(U,-)$ to the above resolution we obtain a complex of finitely supported representations because of Lemma~\ref{lemma.coherentRFS}.
\end{proof}
\end{proposition}

Notice that the associated infinitesimal action of $\Aut_{X}(\sF)$ on $\sE^*$ gives a morphism of Lie algebras
$\Ext^0_X(\sF,\sF)\to Z^0(\Hom_X^*(\sE^*,\sE^*))$; in other words, every endomorphism of $\sF$ lifts functorially to an endomorphism of the complex of sheaves $\sE^*$.

\begin{lemma}\label{lemma.HomClosure}
Let $\sF$ be a coherent sheaf and $\sE$ a locally free sheaf on a projective scheme $X$ over the field $\K$. Then for every open affine subset $U\subset X$, the $\Oh_X(U)$-module 
\[ \Hom_U(\sF,\sE)=\Hom_{\Oh_X(U)}(\sF(U),\sE(U))\]
is a rational representation of $\Aut_{X}(\sF)\times \Aut_{X}(\sE)$ that is finitely supported with respect any subgroup.
\begin{proof}
We have already seen that there exists a sufficiently large $n\in\N$ such that the sheaves $\sF(n)$ and 
$\sE^{\vee}(n)=\HOM_X(\sE,\Oh_X)(n)$ are generated  by global sections, i.e. there exist surjective morphisms of sheaves
\[  H^0(X,\sF(n))\otimes \Oh(-n)\to \sF,\qquad
H^0(X,\sE^{\vee}(n))\otimes \Oh(-n)\to \sE^\vee\to 0\,.\]
Dualizing the second sequence we get an injective morphism of $\Oh_X$-modules 
\[  0\to \sE\to H^0(X,\sE^{\vee}(n))^\vee\otimes \Oh(n)\]
and then for every open affine subset $U\subseteq X$ we get two exact sequences
\[ \Gamma(X,\sF(n))\otimes \Gamma(U,\Oh(-n))\to \sF(U)\to 0,\qquad
0\to\sE(U)\subset H^0(X,\sE^{\vee}(n))^\vee\otimes \Gamma(U,\Oh(n))\,.\]
To conclude it is sufficient to observe that $\Hom_{\Oh_X(U)}(\sF(U),\sE(U))$ is a subrepresentation of $\Hom_{\K}(\sF(U),\sE(U))$ which in turn is a subrepresentation of the rational representation
\[ \begin{split}
&\Hom_{\K}\left(\Gamma(X,\sF(n)))\otimes \Gamma(U,\Oh(-n)), \vphantom{I^1_1}
\Hom_X(\sE,\Oh(n))^\vee\otimes \Gamma(U,\Oh(n))\right)\\
&\qquad=\Hom_{\K}(\Gamma(X,\sF(n)),H^0(X,\sE^{\vee}(n))^\vee)\otimes \Hom_{\K}(\Gamma(U,\Oh(-n)),\Gamma(U,\Oh(n)))\,.
\end{split}\]
\end{proof}
\end{lemma}

\begin{theorem}\label{theorem:goodmodel}
Let $X$ be a projective $\K$-scheme, $\sF$ a coherent sheaf of finite projective dimension on $X$. Consider a resolution $\sE^{\ast}\to \sF$ as in Proposition~\ref{prop.resolution}. Then the derived endomorphisms of $\sF$ are represented by the DG-Lie algebra $\Hom^{\ast}_{A_{\pal}}(C_{\pal}^{\ast}(\sE^{\ast}),C_{\pal}^{\ast}(\sE^{\ast}))$. Moreover, each degree of such complex is a rational representation with respect to the inherited $\Aut_{X}(\sF)$-action, and it is finitely supported for every subgroup.
\begin{proof}
We are only left with the proof of the last part of the statement, since it was proven in Section~\ref{section.derivedEnd} that $\Hom^{\ast}_{A_{\pal}}(C_{\pal}^{\ast}(\sE^{\ast}),C_{\pal}^{\ast}(\sE^{\ast}))$ represents $\RHom(\sF,\sF)$. First notice that for every $k\in\Z$ we have
\[ \Hom^{k}_{A_{\pal}}(C_{\pal}^{\ast}(\sE^{\ast}),C_{\pal}^{\ast}(\sE^{\ast})) \subseteq \prod_{i\in\Z} \Hom_{\K}(C_{\pal}^{i}(\sE^{\ast}),C_{\pal}^{i+k}(\sE^{\ast})) \]
and that by Remark~\ref{rmk.subobjectquotient} and Lemma~\ref{lemma.HomClosure} rational and finitely supported representations are closed under taking subobjects, and homomorphisms. Now recall that $\sE^{\ast}$ is bounded by assumption, and since the scheme is assumed to be projective then the covering can be chosen to be finite, so that $C_{\pal}^{\ast}(\sE^{\ast})$ is bounded too. Hence the product above is finite and the statement follows.
\end{proof}
\end{theorem}

\noindent\textbf{Proof of Theorem~\ref{thm.maintheorem3}.}
Finally, combining the above results we can easily prove Theorem~\ref{thm.maintheorem3}.
To this aim, by Theorem~\ref{theorem:goodmodel} it is sufficient to show that the DG-Lie algebra 
\[L=\Hom^{\ast}_{A_{\pal}}(C_{\pal}^{\ast}(\sE^{\ast}),C_{\pal}^{\ast}(\sE^{\ast}))\] 
satisfies  (2) and (3) of Theorem~\ref{thm.maintheorem3}.
We already noticed that  endomorphisms of $\sF$ lift functorially to endomorphisms of the complex $\sE^*$
and hence of the $A_{\pal}$-module $\Psi\sE^*$. 
Then it is clear from the explicit construction of the cofibrant replacement $C_{\pal}^{\ast}(\sE^{\ast})\to \Psi\sE^*$ 
that  every endomorphism of $\Psi\sE^*$ lifts canonically to an endomorphism  of $C_{\pal}^{\ast}(\sE^{\ast})$.
In conclusion we have an injective morphism of Lie algebras 
\[ \Hom_X(\sF,\sF)=\Ext^0_X(\sF,\sF)\to Z^0(\Hom^{\ast}_{A_{\pal}}(C_{\pal}^{\ast}(\sE^{\ast}),C_{\pal}^{\ast}(\sE^{\ast}))\]
and we can choose its image as the required $\mathfrak{g}$.

\section{Review of $L_\infty[1]$ algebras and formality}	
\subsection{$L_\infty[1]$ algebras} In this subsection we review some basic facts and notations concerning $L_\infty[1]$ algebras, following the paper \cite{Man4}, to which we refer for more details.

Given a graded $\K$-vector space $V=\oplus_{i\in\Z} V^i$, we denote by $S^c(V)=\oplus_{n\geq0}V^{\odot n}$ the symmetric coalgebra over $V$ (see \cite[\S 4]{Man4}). Denoting by $p:S^c(V)\to V$ the natural projection, corestriction induces an isomorphism of graded spaces 
\[\Coder(S^c(V))\to\Hom(S^c(V),V)=\prod_{n\geq0}\Hom(V^{\odot n},V)\colon Q\to p\circ Q= (q_0,q_1,\ldots,q_n,\ldots) \]
(see \cite[Proposition 4.2]{Man4} for an explicit description of the inverse), where $\Coder(S^c(V))$ is the graded Lie algebra of coderivations of $S^c(V)$. We call the components $q_n:V^{\odot n}\to V$ of $Q\in\Coder(S^c(V))$ under the corestriction isomorphisms the \emph{Taylor coefficients} of $Q$, and $q_0$, $q_1$ respectively the \emph{constant} and the \emph{linear part} of $Q$. Via the above isomorphism, the natural commutator bracket on $\Coder(S^c(V))$ induces a Lie bracket on $\Hom(S^c(V),V)$, which is called the \emph{Nijenhuis-Richardson bracket} and denoted by $[-,-]_{NR}$.

\begin{definition} An $L_\infty[1]$ algebra structure on $V$ is a degree $+1$ coderivation
	\[ Q\in\Coder^1(S^c(V)),\qquad q_0=0,\qquad Q\circ Q=0,  \]
	with vanishing constant part and squaring to zero. In particular, the linear part $q_1:V\to V$ squares to zero: the complex $(V,q_1)$ is called the \emph{tangent complex} of the $L_\infty[1]$ algebra $(V,Q)$. The $L_\infty[1]$ algebra $(V,Q)$ is said to be \emph{minimal} if $q_1=0$. 
	
	Given a second $L_\infty[1]$-algebra $(W,R)$, an $L_\infty[1]$ morphism 
	$F\colon V\dashrightarrow W$ from $(V,Q)$ to $(W,R)$ is a morphism of (counital, coaugmented) DG coalgebras $F\colon (S^c(V),Q)\to (S^c(W),R)$. As a morphism of graded (coaugmented, i.e., $F(1)=1$) coalgebras, $F$ is completely determined by its corestriction
	\[ p\circ F = (0,f_1,\ldots,f_n,\ldots)\in\Hom^0(S^c(V),W)=\prod_{n\geq0}\Hom^0(V^{\odot n},W). \]
	As for coderivations, we call the $f_n:V^{\odot n}\to W$ the \emph{Taylor coefficients} of $F$, and $f_1:V\to W$ the \emph{linear part} of $F$. The compatibility with the codifferentials translates into a bunch of identities involving the Taylor coefficients $f_i,q_j,r_k$: we won't need to write these down explicitly. In particular, $f_1:(V,q_1)\to (W,r_1)$ is a morphism between the tangent complexes: we say that $F$ is a \emph{weak equivalence}, if $f_1$ is a quasi-isomorphism. Finally, an $L_\infty[1]$ morphism $F\colon V\dashrightarrow W$ is \emph{strict} if the only non-vanishing Taylor coefficient is the linear one $f_1$.	\end{definition}
\begin{remark} The category of $L_\infty[1]$ algebras is equivalent, up to a shift, to the category of $L_\infty$ algebras, or strong homotopy Lie algebras, see \cite{bibid}, and in particular contains the usual category of DG Lie algebras as a subcategory. More precisely there is a bijective correspondence between $L_\infty[1]$ algebra structures on a graded space $V$ and $L_\infty$ algebra structures on its suspension $V[-1]$: under this correspondence, DG Lie algebra structures correspond to those $L_\infty[1]$ algebra structures $Q$ such that $q_n=0$ for $n\ge3$. 
\end{remark}

\begin{remark}\label{rem: r2} Given an $L_\infty[1]$ algebra $(V,Q)$, the relation $Q\circ Q$ translates into a bunch of relations involving the Taylor coefficients $q_n$, which in particular imply that the quadratic bracket $q_2$ descends to a quadratic bracket 
$r_2\colon H(V,q_1)^{\odot 2}\to H(V,q_1)$ on the tangent cohomology, and the latter satisfies $[r_2,r_2]_{NR}=0$ (in other words, it corresponds to a graded Lie algebra structure on $H(V,q_1)[-1]$).
\end{remark}

A very important and useful fact about $L_\infty[1]$ algebra structures is that they can be transferred along contractions. For a proof of the following result we refer to \cite{HueSta,fuka,BerglundHPT}.
\newcommand{\id}{\operatorname{id}}
\begin{theorem}\label{th: transfer} Let $(V,Q)$ be an $L_\infty[1]$ algebra, and $(f,g,h)$ be a contraction of the tangent complex $(V,q_1)$ onto some complex $(W,r_1)$, i.e., $f:W\to V$ and $g:V\to W$ are DG maps and $h\colon V\to V[-1]$ a contracting homotopy such that the following conditions are satisfied:\[ gf = \id_W,\qquad fg = \id_V + q_1h + hq_1,\qquad gh = h^2 = fh = 0. \]
	Then there is an induced $L_\infty[1]$ structure $R$ on $W$ with linear part $r_1$, together with $L_\infty[1]$ morphisms $F\colon W\dashrightarrow V$ and $G\colon V\dashrightarrow W$ with linear parts $f_1=f$ and $g_1=g$ respectively. The higher Taylor coefficients $r_n,f_n,g_n$, $n\ge2$, are recursively determined by the Taylor coefficients $q_n$, $n\ge2$, and the contraction data $(f,g,h)$ (see e.g. \cite[Theorem 1.9]{descent} for explicit recursive formulas).  
\end{theorem}

An immediate and fundamental consequence of the above theorem is the existence (and, with a little more work, uniqueness) of minimal models.

\begin{definition} Given an $L_\infty[1]$ algebra $(V,Q)$, a \emph{minimal model} of $(V,Q)$ is the datum of a minimal $L_\infty[1]$ algebra $(W,R)$ together with a weak equivalence $F:W\dashrightarrow V$ of $L_\infty[1]$ algebras. General structure theory of $L_\infty[1]$ algebras says that minimal models always exist, and are well defined up to  $L_\infty[1]$ isomorphisms: furthermore, two $L_\infty[1]$ algebras are weakly equivalent if and only if they have isomorphic minimal models. 
\end{definition}

\begin{remark}\label{rem:hodge} In order to obtain an explicit minimal model of $(V,Q)$ it is sufficient to choose an abstract \emph{Hodge decomposition} for the complex $(V,q_1)$. By this we mean the choice of a splitting of the sequence of inclusion $B^i(V)\subset Z^i(V)\subset V^i$, or in other words,  of vector space decompositions $Z^i(V)=B^i(V)\oplus H^i$, $V^i=Z^i(V)\oplus W^i=B^i(V)\oplus H^i\oplus W^i$, for all $i\in\Z$. We notice that the differential $q_1$ vanishes on $H:=\oplus_{i\in\Z}H^i$, and it restricts to an isomorphism from $W[-1]=\oplus_{i\in\Z}W^{i-1}$ to $B(V)=\oplus_{i\in\Z}B^i(V)$. For any such a choice, there is a canonically induced contraction $(f,g,h)$ of $(V,q_1)$ onto $(H,0)$: the inclusion $f:H\to V$ and the projection $g\colon V\to H$ are induced by the decomposition $V=B(V)\oplus H\oplus W$, and the contracting homotopy is the composition $h\colon V\twoheadrightarrow B(V)\xrightarrow{(q_1)^{-1}} W[-1]\hookrightarrow V[-1]$. Via homotopy transfer along this contraction, there is an induced minimal $L_\infty[1]$ algebra structure $(H,R)=(H,0,r_2,\ldots,r_n,\ldots)$ on $H\cong H(V,q_1)$, together with quasi-inverses weak equivalences $F\colon H\dashrightarrow V$, $G\colon V\dashrightarrow H$ such that 
$GF$ is the identity on $H$. We notice that the quadratic bracket $r_2$ identifies with the one from Remark~\ref{rem: r2}.
\end{remark}

An $L_{\infty}[1]$ algebra is called homotopy abelian if it is weakly equivalent to a graded vector space, considered as an $L_{\infty}[1]$ algebra with trivial bracket: it is plain that 
an $L_{\infty}[1]$ algebra is homotopy abelian if and only if its minimal model carries the trivial 
$L_{\infty}[1]$ structure.

\begin{lemma}\label{lem.4sette} 
Let $F\colon (V,q_1,\ldots)\dashrightarrow (W,r_1,\ldots)$ be a morphism of 
$L_{\infty}[1]$ algebra with $W$ homotopy abelian. Then there exists a minimal model 
$(H^*(V,q_1),0,s_2,\ldots)$ for $V$ such that the image of every map $s_r$ is contained
in the kernel of $f_1\colon H^*(V,q_1)\to H^*(W,r_1)$.
\end{lemma} 

\begin{proof} It is not restrictive to assume both $V$ and $W$ minimal, i.e., $q_1=r_1=0$, and then 
$r_n=0$ for every $n$ since  $W$ is assumed homotopy abelian.
Denote by $U\subset W$ the image of $f_1$ and choose a projection $\pi\colon W\to U$. Then 
the composition $\pi F\colon V\dashrightarrow U$ is an $L_{\infty}[1]$ morphism with linear component 
$\pi f_1$ surjective. By general theory of $L_{\infty}[1]$ algebra 
(for a proof see e.g. \cite[Lemma 7.2]{yukawate}) there exists an
$L_{\infty}[1]$ isomorphism $G\colon (H,0,s_2,\ldots) \dashrightarrow V$ such that the composition 
$\pi FG\colon (H,0,s_2,\ldots) \dashrightarrow (U,0,0,\ldots)$ 
is strict and this implies that the image of every $s_r$ is contained in the kernel of the linear part of $\pi FG$.
\end{proof}

We finally come to the main object of interest in this paper, namely, formal $L_\infty[1]$ algebras. 	

\begin{definition} Given an $L_\infty[1]$ algebra $(V,Q)$, we denote by $r_2:H(V,q_1)^{\odot 2} \to H(V,q_1)$ the induced bracket on tangent cohomology, as in Remark \ref{rem: r2}. $V$ is said to be \emph{formal} if there exists a weak equivalence 
\[ F\colon (V,Q)=(V,q_1,q_2,q_3,\ldots,q_n,\ldots)\to(H(V,q_1),r_2)=(H(V,q_1),0,r_2,0,\ldots,0,\ldots)\]
	of $L_\infty[1]$ algebras, or in other words, if every minimal model of $(V,Q)$ is $L_\infty[1]$ isomorphic to $(H(V,q_1),r_2)$.
\end{definition}

\subsection{Formality versus quadraticity} \label{subsec.formalityquadra}
We denote by $\mathbf{Art}_{\K}$ the category of local Artin $\K$-algebras with residue field identified with $\K$. To any $L_\infty[1]$ algebra $V$ is associated a deformation functor $\Def_V\colon\Art_{\K}\to\Set$, sending $A\in\Art_{\K}$ with maximal ideal $\mA$ to the set of Maurer-Cartan elements in $V\otimes\mA$ modulo homotopy equivalence, see \cite{ManRendiconti} for details. Moreover, weakly equivalent $L_\infty[1]$ algebras yield isomorphic deformation functors. According to a well known general philosophy\footnote{Introduced by Nijenhuis and Richardson \cite{NijRich64} in the '60s and sponsored by Deligne, Drinfeld and others during the '80s in the form of private communications \cite{DtM,DtS}, further developed in the works of Goldman-Millson \cite{GoMil1}, Kontsevich \cite{Kont94}, Hinich \cite{hinichdescent,hinichDGC} and Manetti \cite{EDF}, among others, during the '90s, then investigated by Pridham \cite{pridham} and recently studied in the infinity categorical setting by Lurie \cite{lurie2}.}, over a field of characteristic zero every deformation problem is controlled in the above manner by some weak equivalence type of $L_\infty[1]$ algebras, or equivalently, by some isomorphism type of minimal $L_\infty[1]$ algebras. Knowing a controlling $L_\infty[1]$ algebra, it is easy to recover several important features of the deformation problem at hand: for instance, the tangent space $T^1\Def_V$ is isomorphic to $H^0(V,q_1)$, and there is a complete obstruction theory with values in $H^1(V,q_1)$. 

\newcommand{\CE}{\operatorname{CE}}

When $\dim\,H^0(V,q_1)<+\infty$, the deformation functor $\Def_V$ satisfies conditions (H1), (H2) and (H3) from Schlessinger's paper \cite[Theorem 2.11]{Sch}, and in particular, it admits a \emph{hull}. In order to read this off directly from $(V,Q)$ we set $H=H(V,q_1)$ and fix a minimal model $(H,R)$ of $(V,Q)$. We denote by $(H^0)^\vee,(H^1)^\vee$ the dual vector spaces. As $H^0$ is finite dimensional, the Taylor coefficients $r_n:(H^0)^{\odot n}\to H^1$ induce, via transposition, maps $r^\vee_n:(H^1)^{\vee}\to((H^0)^\vee)^{\odot n}$, which together assemble to a map $r^\vee=r_2^\vee+\cdots+r_n^\vee+\cdots:(H^1)^\vee\to\widehat{S}((H^0)^\vee)=\prod_{n\geq0}((H^0)^\vee)^{\odot n}$, where we denote by $\widehat{S}((H^0)^\vee)$ the completed symmetric algebra over $(H^0)^\vee$. If $\dim\,H^0=n$, we can identify $\widehat{S}((H^0)^\vee)$ with the usual algebra of formal power series $\K[[x_1,\ldots,x_n]]$: moreover, we denote by $\mathfrak{m}\subset\K[[x_1,\ldots,x_n]]$ the maximal ideal, and by $I\subset\mathfrak{m}^2$ the ideal generated by the image of $r^\vee$.
\begin{definition} In the above setup, we call the local noetherian complete $\K$-algebra
	\[ A =\frac{\K[[x_1,\ldots,x_n]]}{I} \]
	the \emph{Kuranishi algebra} of $(H,R)$. A bit improperly, we shall also refer to $A$ as the Kuranishi algebra of $(V,Q)$, but notice that it is determined by $(V,Q)$ only up to a non-canonical isomorphism. 
\end{definition}
The following result is essentially shown in \cite{fuka}.
\begin{proposition} Given an $L_\infty[1]$ algebra $(V,Q)$ with $\dim\,H^0(V,q_1)<+\infty$, the associated Kuranishi algebra is a hull for $\Def_V$.
\end{proposition}

We denote by $\widehat{\Art}_{\K}$ the category of local noetherian complete $\K$-algebras with residue field equal to $\K$.
Every such an algebra, up to isomorphism, can be presented 
as a quotient 
\[A=\frac{\K[[x_1,\ldots,x_n]]}{(f_1,\ldots,f_m)}\]
where every $f_i$ has multiplicity $\mu(f_i)\ge 2$. Here 
$n$ is the embedding dimension of $A$ (the dimension of the Zariski tangent space).

\begin{definition}
	An algebra $A\in\widehat{\Art}_{\K}$ is called \emph{quadratic} if it is isomorphic to an algebra of type 
	$\K[[x_1,\ldots,x_n]]/I$, where the ideal $I$ is generated by homogeneous polynomials of degree two. Given an $L_\infty[1]$ algebra $(V,Q)$ such that $\dim\,H^0(V,q_1)<+\infty$, we say that it satisfies the \emph{quadraticity property} if the associated Kuranishi algebra is quadratic.
\end{definition}

Obviously, by homotopy invariance of the Kuranishi algebra, if an $L_\infty[1]$ algebra $(V,Q)$ is formal it satisfies the quadraticity property. The converse is in general not true: formality is a stronger property, putting strong constraints on the full derived deformation functor associated to $V$, while the Kuranishi algebra only remembers the classical part $\Def_V$. 
On the other hand, in several situations the quadraticity property might be easier to verify: for instance, we have the following result, which follows from \cite[Theorem 2.11, Proposition 2.14 and Theorem 2.16]{MartPadova}.

\begin{proposition} Let $F:(V,Q)\dashrightarrow(W,R)$ be and $L_\infty[1]$ morphism, and assume that $H^0(f_1):H^0(V,q_1)\to H^0(W,r_1)$ is surjective, $H^1(f_1):H^1(V,q_1)\to H^1(W,r_1)$ is injective and $\dim\, H^0(V,q_1)<+\infty$. Then the $L_\infty[1]$ algebra $(V,Q)$ satisfies the quadraticity property if and only if so does $(W,R)$.
\end{proposition}

\begin{remark} The above proposition fails if in its statement we replace the quadraticity property by the property of being formal: to the best of the authors' knowledge,  the only result one can find in the literature going in a somewhat similar direction is the formality transfer theorem from \cite[Theorem 6.8]{Man4}, whose hypotheses are much harder to verify. With respect to the discussion in the introduction of \cite{budur}, this is essentially the reason why the quadraticity conjecture \cite[Conjecture 1.1]{budur} by Kaledin and Lehn \cite{KaLe} has been easier to handle than the full formality conjecture \cite[Conjecture 1.2]{budur}.
\end{remark}

Our aim in the remainder of this subsection is to show that in the special case when $H(V,q_1)=:H=H^0\oplus H^1$ is concentrated in degrees zero and one and $\dim\,H=\dim\,H^0+\dim\,H^1<+\infty$, then $V$ is formal if and only if the associated Kuranishi algebra is quadratic (the assumption $\dim\,H^1<+\infty$ is actually unnecessary, we keep it for simplicity and since it is usually satisfied in concrete examples).

First, we shall look more closely at quadratic algebras in $\widehat{\Art}_\K$.

\textbf{Notation.} For every $f\in \K[[x_1,\ldots,x_n]]$ we denote by $\mu(f)$ its multiplicity, and by 
$f^{(n)}$ its homogeneous component of degree $n$, hence $f=f^{(\mu(f))}+f^{(\mu(f)+1)}+\cdots$.

\begin{lemma} Let 
	\[A=\frac{\K[[x_1,\ldots,x_n]]}{(f_1,\ldots,f_m)},\qquad \mu(f_i)\ge 2,\]
	be a quadratic algebra. Then there exists an isomorphism 
	\[ \phi\colon \K[[x_1,\ldots,x_n]]\to \K[[x_1,\ldots,x_n]]\]
	with differential (i.e. the linear part) equal to the identity  such that 
	the ideal
	\[ (\phi(f_1),\ldots,\phi(f_m))\]
	is generated by the quadrics  $\phi(f_i)^{(2)}=f_i^{(2)}$, $i=1,\ldots,m$.
\end{lemma}

\begin{proof} By assumption there exists 
	an isomorphism 
	\[ \psi\colon \K[[x_1,\ldots,x_n]]\to \K[[x_1,\ldots,x_n]]\]
	such that the ideal $\psi(f_1,\ldots,f_m)$ is generated by quadrics. Define 
	$\phi=\psi_1^{-1}\psi$, where $\psi_1$ is the  automorphism induced by the linear part of 
	$\psi$. Then $(\phi(f_1),\ldots,\phi(f_m))$ is generated by quadrics:
	\[ (\phi(f_1),\ldots,\phi(f_m))=(q_1,\ldots,q_r),\qquad q_i=q_i^{(2)}\,.\]
	Since $\mu(f_i)\ge 2$ for every $i$, 
	every $q_i$ is a linear combination of $f_1^{(2)},\ldots,f_m^{(2)}$ and conversely.
	Thus $f_1^{(2)},\ldots,f_m^{(2)}$ and   $q_1,\ldots,q_r$ generate the same vector space and therefore also the same ideal.
\end{proof}

For simplicity of notation, we denote by $P=\K[[x_1,\ldots,x_n]]$ and by $\mathfrak{m}\subset P$ its maximal ideal.

\begin{lemma}
	With the above notations, let $f_1,\ldots,f_m\in \mathfrak{m}^2$ such that the ideal 
	$(f_1,\ldots,f_n)$ is generated by quadrics. Then there exists a matrix 
	$A=(a_{ij})\in M_{m,m}(P)$ such that $(a_{ij})\equiv \Id$ $\pmod{\mathfrak{m}}$, and
	\[ \sum_{j=1}^m a_{ij}\,f_j^{(2)}=f_i,\qquad i=1,\ldots,m\,.\]
\end{lemma}

\begin{proof} There exists an invertible matrix $C=(c_{ij})\in M_{m,m}(\K)$  such that the power series 
	$g_i=\sum_{j=1}^m c_{ij}\,f_j$ have the property that 
	$g_1^{(2)},\ldots,g_r^{(2)}$ are linearly independent over $\K$ and 
	$g_{r+1}^{(2)}=\cdots=g_m^{(2)}=0$ for some $0\le r\le m$.
	
	Thus we have $(g_1,\ldots,g_m)=(g_1^{(2)},\ldots,g_r^{(2)})$ and for every $i=1,\ldots,m$ there exists power series $b_{i1},\ldots,b_{ir}\in P$ such that 
	\[ \sum_{j=1}^r b_{ij}g_j^{(2)}=g_i.\,\]
	In particular 
	\[ \sum_{j=1}^r b_{ij}(0)g_j^{(2)}=g_i^{(2)}\]
	and then $b_{ij}(0)=1$ if $i=j\le r$ and 
	$b_{ij}(0)=0$ otherwise.
	For every $j>r$ define $b_{ij}=1$ if $i=j$ and $b_{ij}=0$ if $i\not=j$. 
	Then the matrix $B=(b_{ij})\in M_{m,m}(P)$  is such that 
	$(b_{ij}(0))=\Id$  and
	\[ \sum_{j=1}^m b_{ij}\,g_j^{(2)}=g_i,\qquad i=1,\ldots,m\,.\]
	Finally take $A=C^{-1}BC$.
\end{proof}

\begin{proposition}\label{prop:quad} Notation as above. For a sequence 
	$f_1,\ldots,f_m\in \mathfrak{m}^2$ the following conditions are equivalent:
	
	\begin{enumerate}
		
		\item the algebra $P/(f_1,\ldots,f_m)$ is quadratic;
		
		\item there exists an isomorphism of algebras $\phi\colon P\to P$ with linear part equal the identity and a 
		matrix $(a_{ij})\in M_{m,m}(P)$ with $(a_{ij}(0))=\Id$ such that 
		\[ \sum_{j=1}^m a_{ij}\,f_j^{(2)}=\phi(f_i)\,.\]
	\end{enumerate}
\end{proposition}

\begin{proof} Immediate from lemmas.
\end{proof}

Next, we shall look more closely at the category of minimal $L_\infty[1]$ algebras $(H=H^0\oplus H^1,R)$ concentrated in degrees zero and one and of finite total dimension $\dim\,H=\dim\,H^0+\dim\,H^1<+\infty$. Since $\dim\,H<+\infty$, denoting by $H^\vee$ the graded dual ($(H^\vee)^i=(H^{-i})^{\vee}$), transposition induces an anti-isomorphism of graded Lie algebras
\[ \Coder(S^c(H))\cong\prod_{n\ge0} \Hom(H^{\odot n},H) \xrightarrow{\cong} \prod_{n\geq0} \Hom(H^\vee,(H^\vee)^{\odot n})\cong\Der(\widehat{S}(H^\vee))   \]
from $\Coder(S^c(H))$ to the graded Lie algebra of derivations of the completed symmetric algebra $\widehat{S}(H^\vee)=\prod_{n\geq0}(H^\vee)^{\odot n}$ over $H^\vee$. In particular, the $L_\infty[1]$ algebra structure $R$ on $H$ corresponds, under transposition, to a DG algebra structure $R^\vee$ on $\widehat{S}(H^\vee)$: we call the DG algebra $(\widehat{S}(H^\vee),R^\vee)$ the \emph{Chevalley-Eilenberg} algebra of $(H,R)$, and denote it by $\CE(H,R)^\vee$. In this context, the Kuranishi algebra of $(H,R)$ is just $H^0(\operatorname{CE}(H,R)^\vee,R^\vee)$. 

We fix bases $e_1,\ldots,e_n$ and $\epsilon_1,\ldots,\epsilon_m$ of $H^0$ and $H^1$ respectively, together with the dual bases $x_1,\ldots,x_n$ and $\eta_1,\ldots,\eta_m$ of $(H^0)^\vee$ and $(H^1)^\vee$. We shall denote the graded algebra $\widehat{S}(H^\vee)$ by $\mathbb{K}[[x_1,\ldots,x_n,\eta_1,\ldots,\eta_m]]$, that is, the free complete graded commutative algebra generated by $x_1,\ldots,x_n$ in degree zero and $\eta_1,\ldots,\eta_m$ in degree minus one: in particular, its degree zero component is the usual algebra $\mathbb{K}[[x_1,\ldots,x_n]]$ of formal power series. As before, we denote by $\mathfrak{m}\subset\mathbb{K}[[x_1,\ldots,x_n]]$ the maximal ideal. The DG algebra structure $R^\vee$ on $\mathbb{K}[[x_1,\ldots,x_n,\eta_1,\ldots,\eta_m]]$ is completely determined by the $m$-tuple of formal power series $R^\vee(\eta_1)=:f_1,\ldots,R^\vee(\eta_m)=:f_m$, and since $R$ is a minimal $L_\infty[1]$ algebra structure $f_1,\ldots,f_m\in\mathfrak{m}^2$. More explicitly, if we denote by $I=(i_1,\ldots,i_n)\in\mathbb{N}^n$ a multi-index, by $x^I:=x_1^{i_1}\cdots x_n^{i_n}$, $e_I:=e_1^{\odot i_1}\odot\cdots\odot e_n^{\odot i_n}$ and $I!=i_1!\cdots i_n!$, if the corestriction $r\in\Hom(S^c(H^0),H^1)$ of $R$ is given by $r(e_I)=\sum_{j=1}^m r_I^j\epsilon_j$, then the corresponding formal power series are given by $f_j=\sum_{I\in\mathbb{N}^n}\frac{r_I^j}{I!}x^I$. 

Let $(K=K^0\oplus K^1,Q)$ be another finite dimensional minimal $L_\infty[1]$ algebra in degrees zero and one. We fix bases $e'_1,\ldots,e'_s$ and $\epsilon'_1,\ldots,\epsilon'_r$ of $K^0$ and $K^1$ respectively, together with the dual bases $y_1,\ldots,y_s$ and $\theta_1,\ldots,\theta_r$ of $(K^0)^\vee$ and $(K^1)^\vee$. Let the corresponding DG algebra structure on $\mathbb{K}[[y_1,\ldots,y_s,\theta_1,\ldots,\theta_r]]$ be given by $Q^\vee(\theta_1)=g_1,\ldots, Q^\vee(\theta_r)=g_r\in\mathbb{K}[[y_1,\ldots,y_s]]$. A morphism of graded coalgebras $F:S^c(H)\to S^c(K)$ corresponds dually to a morphism of graded algebras 
\[ F^\vee: \mathbb{K}[[y_1,\ldots,y_s,\theta_1,\ldots,\theta_r]] \to \mathbb{K}[[x_1,\ldots,x_n,\eta_1,\ldots,\eta_m]] \]
The latter is completely determined by its restriction to the degree zero components, which we denote by $\phi:\mathbb{K}[[y_1,\ldots,y_s]]\to\mathbb{K}[[x_1,\ldots,x_n]]$, and the formal power series $a_{ij}\in\mathbb{K}[[x_1,\ldots,x_n]]$ defined by $F^\vee(\theta_i)=\sum_{j=1}^m a_{ij}\eta_j$, $i=1,\ldots,r$. The requirement that $F$ in an $L_\infty[1]$ morphism, that is, $F^\vee$ is a morphism of DG algebras, becomes 
\begin{equation}\label{eq:morp} \phi(g_i) = F^\vee Q^\vee(\theta_i) =  R^\vee F^\vee(\theta_i) = R^\vee\left(\sum_{j=1}^m a_{ij}\eta_j\right) = \sum_{j=1}^m a_{ij}f_j,\qquad\forall\, i=1,\ldots,r.
\end{equation}

Putting together the previous considerations, the desired result follows easily.

\begin{theorem}\label{thm.4sedici} 
Let $(V,Q)$ be an $L_\infty[1]$ algebra with finite dimensional tangent cohomology $H(V,q_1)=:H=H^0\oplus H^1$ concentrated in degrees zero and one. Then $V$ is formal if and only if it satisfies the quadraticity property.
\end{theorem}
\begin{proof} The only if implication is clear. Conversely, let $(H,R)\dashrightarrow (V,Q)$ be a minimal model of $(V,Q)$. As before, we fix bases of $H^0, H^1$, and the dual bases $x_1\ldots,x_n$, $\eta_1,\ldots,\eta_m$ of $(H^0)^\vee$, $(H^1)^\vee$ respectively: then the $L_\infty[1]$ algebra structure $R$ is determined by the formal power series $f_i=R^\vee(\eta_i)\in\mathfrak{m}^2$, $i=1,\ldots,m$, and the Kuranishi algebra of $(V,Q)$ is
	\[\frac{\K[[x_1,\ldots,x_n]]}{(f_1,\ldots,f_m)}.\]
	If the latter is quadratic, we can construct $\phi:\K[[x_1,\ldots,x_n]]\to\K[[x_1,\ldots,x_n]]$ and $A=(a_{ij})$ as in Proposition \ref{prop:quad}. According to the discussion preceding Equation \eqref{eq:morp}, the datum of $\phi$ and $A$ is equivalent to the datum of an $L_\infty[1]$ isomorphism
	\[ (H,R)=(H,0,r_2,r_3.\ldots,r_k,\ldots)\dashrightarrow(H,r_2)=(H,0,r_2,0,\ldots,0,\ldots) \]
	with linear part the identity, hence $(V,Q)$ is formal.

\end{proof}

\begin{corollary}\label{cor.4sedici} 	
Let $(V,Q)$ be an $L_\infty[1]$ algebra with tangent cohomology in nonnegative degree.  
Assume that:
\begin{enumerate}

\item $H^0(V,q_1)$ and $H^1(V,q_1)$ are finite dimensional vector spaces;

\item there exists an $L_{\infty}$-morphism $f\colon (V,Q)\dashrightarrow (W,R)$ into a homotopy abelian $L_\infty[1]$ algebra $(W,R)$ such that $f\colon H^i(V,q_1)\to H^i(W,r_1)$ is injective for every $i\ge 2$.
\end{enumerate}
Then $V$ is formal if and only if it satisfies the quadraticity property.
\end{corollary}

\begin{proof} By Lemma~\ref{lem.4sette} there exists a minimal model $(H,0,s_2,\ldots)$ such that 
$H^i=0$ for every $i<0$ and the image of $s_r$ is contained in $H^0\oplus H^1$ for every $r$.

Hence $H$ is the direct product of two $L_{\infty}[1]$ algebras, namely 
\[(H^0\oplus H^1,0,s_2,\ldots)\times (\oplus_{i\ge 2}H^i,0,0,\ldots)\,.\] 
Thus if $H^0\oplus H^1$ is formal then 
also $H$ and $V$ are formal; on the other side $V$, $H$ and $H^0\oplus H^1$ have isomorphic  
Kuranishi algebras and the conclusion follows by Theorem~\ref{thm.4sedici}.
\end{proof}

\subsection{Equivariant formality} Let $(g,v)\to gv$ be a representation of a group $G$ on a graded vector space $V$. 

\begin{definition} In the above setup, an $L_\infty[1]$ algebra structure $Q$ on $V$ is \emph{$G$-equivariant} if so are its Taylor coefficients, i.e.,
	\[ gq_n(v_1,\ldots,v_n)=q_n(gv_1,\ldots,gv_n),\qquad\forall\,g\in G,\, n\ge1,\,v_1,\ldots,v_n\in V. \]
	Similarly, given a second $G$-equivariant $L_\infty[1]$ algebra $(W,R)$, an $L_\infty[1]$ morphism $F:(V,Q)\dashrightarrow(W,R)$ is $G$-equivariant if
	\[ gf_n(v_1,\ldots,v_n)=f_n(gv_1,\ldots,gv_n),\qquad\forall\,g\in G,\, n\ge1,\,v_1,\ldots,v_n\in V. \]
\end{definition}

Given a $G$-equivariant $L_\infty[1]$ algebra $(V,Q)$, the $G$-action descends on the tangent cohomology $H(V,q_1)$, and the induced bracket $r_2\colon H(V,q_1)^{\odot2}\to H(V,q_1)$, as in Remark \ref{rem: r2}, is $G$-equivariant. We say that $(V,Q)$ is \emph{$G$-equivariantly formal} if there exists a $G$-equivariant weak equivalence $(H(V,q_1),r_2)\dashrightarrow(V,Q)$. The aim of this subsection is to prove that under some mild assumptions a $G$-equivariant $L_\infty[1]$ algebra is $G$-equivariantly formal if and only if it is formal in the usual sense.
\begin{theorem}\label{th:Gformality} Let $(V,Q)$ be a $G$-equivariant $L_\infty[1]$ algebra. Assume that the following $G$-modules are semisimple:\begin{itemize}
		\item $V^i$, for all $i\in\Z$; and\item $\Hom^j(H(V,q_1)^{\odot n}, H(V,q_1) )$ for $j=0,1$ and all $n\ge2$.
	\end{itemize}
	Then, if $V$ is formal it is also $G$-equivariantly formal.
\end{theorem}

\begin{remark}\label{rem:Gformality} For instance, the hypotheses of the theorem are automatically satisfied in the following situations:\begin{itemize} \item when $G$ is finite;\item when $G$ is a linearly reductive algebraic group and $V^i$ is a finitely supported rational $G$-module for all $i\in\Z$;
		\item when $G$ is a linearly reductive algebraic group, $V^i$ is a rational $G$-module for all $i\in\Z$ and moreover $\dim H(V,q_1)=\sum_{i\in\Z}\dim\,H^i(V,q_1) <+\infty$.
	\end{itemize}
	The first item follows from the well known fact that any representation of a finite group is semisimple. The second and third items from the fact that, since $G$ is assumed linearly reductive, any rational representation of $G$ is semisimple, and the class of rational representations (as well as the subclass of finitely supported rational representations) is closed under subobjects and quotients, see Remark~\ref{rmk.subobjectquotient}.  Moreover, the class of finite dimensional rational representations is further closed under tensor products and Hom spaces, and Lemma~\ref{lemma.finsup} shows that the same is true for the class of finitely supported rational representations.
\end{remark}

\begin{proof} (of Theorem \ref{th:Gformality}.) The proof is broken into two steps: the first step is to show the existence of a $G$-equivariant minimal model of $(V,Q)$. The assumption that $V^i$ is semisimple, $\forall\,i\in\Z$, assures the existence of a $G$-invariant abstract Hodge decomposition of the complex $(V,q_1)$ (cf. with Remark \ref{rem:hodge}), that is, direct sum decompositions $Z^i(V)=B^i(V)\oplus H^i$, $V^i=Z^i(V)\oplus W^i=B^i(V)\oplus H^i\oplus W^i$ such that $H^i,W^i\subset V^i$ are $G$-invariant subspaces. It follows that $H=\oplus_{i\in\Z} H^i$ is a $G$-module, and the induced contraction $(f,g,h)$ of $(V,q_1)$ onto $(H,0)$, as in Remark \ref{rem:hodge}, is $G$-equivariant. It is now a straightforward consequence of the explicit recursive formulas for homotopy transfer (for which, once again, we may refer to \cite[Theorem 1.9]{descent}) that the induced minimal $L_\infty[1]$ algebra structure $R$ on $H$ and the induced $L_\infty[1]$ weak equivalences $F:(H,,R)\dashrightarrow(V,Q)$ and $G:(V,Q)\dashrightarrow(H,R)$ are all $G$-equivariant. 
	
	The second step is to prove that if $(V,Q)$ is formal, then $(H,R)$ constructed as in the previous paragraph is $G$-equivariantly $L_\infty[1]$ isomorphic to $(H,r_2)$: recall that formality of $(V,Q)$ is equivalent to the existence of an $L_\infty[1]$ isomorphism $(H,R)\dashrightarrow(H,r_2)$ with linear part the identity, but a priori the latter might not be $G$-equivariant. We reason inductively, following closely the proof of \cite[Theorem 6.3]{Man4}: the inductive step will depend on the following key lemma, which is a $G$-equivariant version of \cite[Lemma 6.2]{Man4}.
	
	\begin{lemma}\label{lem:man4} Given $k\geq3$ and a minimal $G$-equivariant $L_\infty[1]$ algebra of the form $(H,R')=(H,0,r'_2,0,\ldots,0,r'_k,r'_{k+1},\ldots)$, if $(H,R')$ is formal in the usual sense and the $G$-module $\Hom^j(H^{\odot n}, H)$ is semisimple  for $j=0,1$ and $n=k-1,k$, then there exists a degree zero $G$-equivariant map $\alpha_{k-1}:H^{\odot k-1}\to H$ such that $r'_k=[r'_2,\alpha_{k-1}]_{NR}$.
	\end{lemma}
	\begin{proof} According to \cite[Lemma 6.2 (3)]{Man4} there exists a not necessarily $G$-equivariant map $\widetilde{\alpha}_{k-1}:H^{\odot k-1}\to H$ having the desired property $[r'_2,\widetilde{\alpha}_{k-1}]_{NR}=r'_k$. Furthermore, the hypothesis that $\Hom^j(H^{\odot n},H)$ is semisimple ($j=0,1$, $n=k-1,k$) implies the existence of \emph{Reynolds operators} (cf. \cite[\S 6.2.4]{procesi}), that is, projectors $R\colon \Hom^j(H^{\odot n},H)\twoheadrightarrow\Hom^j(H^{\odot n}, H)^G$ onto the invariants, which are natural with respect to maps of $G$-modules (see \cite[\S 6.2.3, Proposition 2]{procesi}). Finally, we notice that the Nijenhuis-Richardson bracket $[-,-]_{NR}$ is equivariant with respect to the natural action of $G$ on $\Hom^j(H^{\odot n}, H)$ given by $(gq_n)(h_1,\ldots,h_n)= gq_n(g^{-1}h_1,\ldots,g^{-1}h_n)$. Since $r'_2$ is supposed to be $G$-equivariant, $[r'_2,-]_{NR}:\Hom^0(H^{\odot k-1},H)\to \Hom^1(H^{\odot k},H)$ is a morphism of $G$-modules, and then the diagram
		\[ \xymatrix{ \Hom^0(H^{\odot k-1},H)\ar[d]_{[r'_2,-]_{NR}}\ar[r]^-R & \Hom^0(H^{\odot k-1},H)^G\ar[d]_{[r'_2,-]_{NR}} \\ \Hom^1(H^{\odot k},H)\ar[r]^-R & \Hom^1(H^{\odot k},H)^G}   \]
		is commutative. In particular, since $r'_k$ is also supposed to be $G$-equivariant, we have
		\[ r'_k = R(r'_k) = R([r'_2,\widetilde{\alpha}_{k-1}]_{NR}) = [r'_2, R(\widetilde{\alpha}_{k-1})]_{NR}, \]
		and the proof is concluded by setting $\alpha_{k-1}=R(\widetilde{\alpha}_{k-1})$.
	\end{proof}
	
	Going back to the proof of the theorem, given $k\geq3$ we assume inductively to have constructed a minimal $L_\infty[1]$ algebra $(H,R')=(H,0,r'_2,0,\ldots,0,r'_{k},r'_{k+1},\ldots)$ as in the hypotheses of the previous lemma and a $G$-equivariant $L_\infty[1]$-isomorphism $F':(H,R)\to(H,R')$ with linear part the identity. In particular, this implies $r'_2=r_2$. The basis of the induction is $k=3$, $R'=R$ and $(f'_1,f'_2,\ldots,f'_n,\ldots)=(\id_H,0,\ldots,0,\ldots)$. Since $(H,R)$ is supposed to be formal, so is $(H,R')$, and we can find $\alpha_{k-1}$ as in Lemma \ref{lem:man4}. As in the proof of \cite[Theorem 6.3]{Man4}, we denote by $\widehat{\alpha}_{k-1}$ the $G$-equivariant coderivation $\widehat{\alpha}_{k-1}=(0,0,\ldots,0,\alpha_{k-1},0,\ldots)$: its exponential $e^{\widehat{\alpha}_{k-1}}$ is a well defined $G$-equivariant automorphism of the symmetric coalgebra $S^c(H)$, acting as the identity on $\oplus_{i<k-1}H^{\odot i}\subset S^c(H)$. Finally, we put $R'' = e^{\widehat{\alpha}_{k-1}}\circ R'\circ e^{-\widehat{\alpha}_{k-1}}$ and $F''=e^{\widehat{\alpha}_{k-1}}\circ F'$. It is clear that $F'':(H,R)\to(H,R'')$ is an $L_\infty[1]$ isomorphism with linear part the identity, and the same computations as in the proof of \cite[Theorem 6.3]{Man4} show that $R''$ has the form $R''=(0,r_2,0,\ldots,0,r''_{k+1},r''_{k+2},\ldots)$: moreover, since $R'$, $F'$ and $e^{\widehat{\alpha}_{k-1}}$ are $G$-equivariant, so are $R''$ and $F''$, and we can proceed with the induction.
	
	In order to conclude, it is sufficient to observe that the infinite composition 
	\[ F:=\cdots\circ e^{\widehat{\alpha}_{k-1}}\circ\cdots\circ e^{\widehat{\alpha}_{3}}\circ e^{\widehat{\alpha}_{2}}\colon S^c(H)\to S^c(H) \]
	is well defined, since $e^{\widehat{\alpha}_{k-1}}$ acts as the identity on $\oplus_{i<k-1}H^{\odot i}\subset S^c(H)$, and by construction $F$ is an $L_\infty[1]$ isomorphism 
$F\colon(H,0,r_2,r_3,\ldots,r_n,\ldots)\dashrightarrow(H,0,r_2,0,\ldots,0,\ldots)$ with linear part the identity.
\end{proof}

\begin{corollary}\label{cor.Gformality} 
Let $(L,d,[-,-])$ be a DG Lie algebra. Assume the following hypotheses:\begin{itemize}
		\item $H^i(L,d)=0$ for $i<0$ and $H^0(L,d)=:\mathfrak{g}$ is the Lie algebra of a reductive algebraic group $G$;
		\item there is a Lie algebra embedding $\imath\colon\mathfrak{g}\hookrightarrow Z^0(L)$ that is a section of the natural projection map $Z^0(L)\to H^0(L,d)$;
		\item the adjoint action of $\mathfrak{g}$ on $L$ is induced by an action of $G$ by DG Lie algebra automorphisms which is degreewise rational and finitely supported.
	\end{itemize}
	Define a DG Lie subalgebra $K\subset L$ by setting $K^i=0$ for $i\leq0$, $K^1\subset L^1$ a $\mathfrak{g}$-invariant complement of $B^1(L)$ in $L^1$ (this exists because the hypotheses imply that $L^i$ is a semisimple $\mathfrak{g}$-module for all $i\in\Z$) and finally $K^i=L^i$ for $i\geq2$. Then, the DG Lie algebra $L$ is formal if and only if so is the one $K$.
\end{corollary}
\begin{proof} We introduce a second DG Lie subalgebra $M:=\imath(\mathfrak{g})\oplus K\subset L$: clearly the inclusion $M\hookrightarrow L$ is a quasi-isomorphism, thus $L$ is formal if and only if so is $M$, and we need to prove that this is the case if and only if so is $K$.
	
We denote the $L_\infty[1]$ algebras corresponding to $M$ and $K$ respectively by \[ (M[1],Q)=(L[1],q_1,q_2,0,\ldots,0,\ldots)\qquad\mbox{and}\qquad(K[1],Q)=(K[1],q_1,q_2,0,\ldots,0,\ldots),\] 
and by $(H(M,d)[1],r_2)$ and $(H(K,d)[1],r_2)$ the $L_\infty[1]$ algebras corresponding to their cohomology graded Lie algebras (explicitly, if $M\to M[1]: x\to s^{-1}x$ is the desuspension isomorphism and $|x|$ is the degree of $x$ in $M$, then $q_1(s^{-1}x)=-s^{-1}dx$, $q_2(s^{-1}x\odot s^{-1}y)=(-1)^{|x|}s^{-1}[x,y]$: similarly for $r_2$). We notice that by construction $H(K,d)=H^{>0}(M,d)=\oplus_{i>0} H^i(M,d)$.

If $M$ is formal and $F\colon (H(M,d)[1],r_2)\dashrightarrow (M[1],Q)$ is an $L_\infty[1]$ weak equivalence, for trivial degree reasons when $x_1,\ldots,x_n\in H(K,d)[1]=H^{>0}(M,d)[1]\subset H(M,d)[1]$ then $f_n(x_1,\ldots,x_n)\in M^{>0}[1]=K[1]$, that is, $F$ restricts to an $L_\infty[1]$ weak equivalence 
$F\colon(H(K,d)[1],r_2)\dashrightarrow(K[1],Q)$ and $K$ is formal.  
	
Conversely, we assume that $K$ is formal. By the hypotheses, $(K[1],Q)$ is a $G$-equivariant $L_\infty[1]$ algebra satisfying the assumptions of Theorem \ref{th:Gformality} (see also Remark \ref{rem:Gformality}), hence it is also $G$-equivariantly formal. Moreover, we fix a $G$-equivariant $L_\infty[1]$ weak equivalence $\widetilde{F}:(H(K,d)[1],r_2)\dashrightarrow(K[1],Q)$, and then we can define an $L_\infty[1]$ weak equivalence $F:(H(M,d)[1],r_2)\dashrightarrow(M[1],Q)$ as follows. Using the natural identification $H(M,d)\cong\mathfrak{g}\oplus H(K,d)$, we put $f_1(x)=\widetilde{f}_1(x)$ if $x\in H(K,d)[1]\subset H(M,d)[1]$ and $f_1(x)=\imath(x)$ if $x\in\mathfrak{g}[1]\subset H(M,d)[1]$: for $n\geq2$, we put $f_n(x_1,\ldots,x_n)=\widetilde{f}_n(x_1,\ldots,x_n)$ if $x_1,\ldots,x_n\in H(K,d)[1]\subset H(M,d)[1]$, and $f_n(x_1,\ldots,x_n)=0$ otherwise. It is clear that $f_1$ is a quasi-isomorphism. The remaining relations that have to be satisfied in order for $F$ to define an $L_\infty[1]$ morphism read
	\begin{multline}\label{eq:Loomorp} q_1f_n(x_1,\ldots,x_n) +\frac{1}{2}\sum_{i=1}^{n-1}\sum_{\sigma\in S(i,n-i)}\pm q_2(f_i(x_{\sigma(1)},\ldots,x_{\sigma(i)}),f_{n-i}(x_{\sigma(i+1)},\ldots,x_{\sigma(n)})) = \\ \sum_{\sigma\in S(2,n-2)}\pm f_{n-1}(r_2(x_{\sigma(1)},x_{\sigma(2)}),x_{\sigma(3)},\ldots,x_{\sigma(n)})
	\end{multline}
	for all $n\geq2$ and $x_1,\ldots,x_n\in H(M,d)[1]$, where $\pm$ is the appropriate Koszul sign and $S(i,n-i)$ is the set of $(i,n-i)$-unshuffles (i.e., permutations $\sigma\in S_n$ such that $\sigma(j)<\sigma(j+1)$ for $j\neq i$). When $x_1,\ldots,x_n\in H(K,d)[1]\subset H(M,d)[1]$ these relations follow from the corresponding ones for the $\widetilde{f}_n$, whereas both sides of \eqref{eq:Loomorp} are obviously zero whenever $x_i,x_j\in\mathfrak{g}[1]\subset H(M,d)[1]$ for some $i\neq j$ and $n\geq3$. For $n=2$ and $x_1,x_2\in\mathfrak{g}[1]$ \eqref{eq:Loomorp} follows from the fact that 
	$\imath\colon \mathfrak{g}\to Z^0(L)$ is a morphism of Lie algebras. In the remaining case, it is not restrictive to assume $x_1\in\mathfrak{g}[1]$ and $x_2,\ldots,x_n\in H(K,d)[1]$: in this situation \eqref{eq:Loomorp} reduces to
	\[ q_2(\imath(x_1),\widetilde{f}_{n-1}(x_2,\ldots,x_n)) = \sum_{j=2}^{n}\pm\widetilde{f}_{n-1}(r_2(x_1,x_j),x_2,\ldots,\widehat{x_j},\ldots,x_n), \]
	which says that the Taylor coefficients 
	$\widetilde{f}_{n-1}\colon H(K,d)[1]^{\odot n-1}\to K[1]$ are $\mathfrak{g}$-equivariant maps. This is clear, since by assumption the $\mathfrak{g}$-module structures are induced by $G$-module structures and the $\widetilde{f}_{n-1}$ are $G$-equivariant.
\end{proof}

\noindent\textbf{Proof of Theorem~\ref{thm.maintheorem2}.}
We only need to prove that if the Kuranishi family of $\sF$ is quadratic then $\RHom(\sF,\sF)$ is formal, since the converse is true by general theory.

In the situation of Corollary~\ref{cor.Gformality}, the DG-Lie algebras $L$ and $K$ have the same Kuranishi family 
\cite{GoMil2,ManRendiconti,NijRich64}.  
If $H^i(L)=0$ for every $i\not=0,1,2$, then $H^i(K)=0$ for every $i\not=1,2$. Therefore by 
Corollary~\ref{cor.Gformality} $L$ is formal if and only if so is  $K$, 
while by  
Corollary~\ref{cor.4sedici} $K$ is formal if  its  Kuranishi algebra is quadratic and there exists an $L_{\infty}$ morphism from $K$ to a homotopy abelian $L_{\infty}$ algebra that is injective
on the tangent cohomology groups $H^i$ for every $i\ge 3$. 

According to \cite{DMcoppie} the trace maps $\Tr\colon \Ext^i_X(\sF,\sF)\to \Ext^i_X(\Oh_X,\Oh_X)=H^i(\Oh_X)$ are induced by an $L_{\infty}$ morphism $\Tr\colon \RHom(\sF,\sF)\to 
\RHom(\Oh_X,\Oh_X)$ and  $\RHom(\Oh_X,\Oh_X)$ may be represented by the  
abelian DG-Lie algebra given by the Thom-Whitney totalisation of the 
sheaf of abelian Lie algebras $\HOM_{\Oh_X}(\Oh_X,\Oh_X)$ with respect to a given affine 
open cover, see also \cite[pag. 2255]{FMM}. We also refer to \cite{FM} for an explicit 
description of a quasi-isomomorphism between the Thom-Whitney totalisations and the representative
of derived endomorphisms described here in Subsection~\ref{section.derivedEnd}.

Now Theorem~\ref{thm.maintheorem2} follows immediately from the previous considerations applied to the DG-Lie algebra of  Theorem~\ref{thm.maintheorem3}

\begin{remark} Theorem~\ref{thm.maintheorem1} holds in particular for every polystable sheaf $\sF$ on a smooth projective surface $X$ such that  
$\Ext^2_X(\sF,\sF)=0$, since the Kuranishi family is smooth and therefore quadratic.
For instance, according to \cite[Cor. 6.7.3]{maru} the equality  $\Ext^2_X(\sF,\sF)=0$ holds whenever 
$\sF$ is a torsion free $H$-semistable sheaf and $H\cdot K_X<0$. 
\end{remark}

\textbf{Acknowledgements:} we are grateful to the referee for carefully reading the manuscript and for useful comments on the first version of this paper.

\end{document}